\documentclass{article}
\usepackage{amsmath,amssymb,amsthm}
\usepackage{accents}
\usepackage{color}
\usepackage[top=25truemm,bottom=25truemm,left=25truemm,right=25truemm]{geometry}
\usepackage{graphicx}
\usepackage{hyperref}
\usepackage{mathtools}
\theoremstyle{definition}
\newtheorem{Def}{Definition}[section]
\newtheorem{Lem}[Def]{Lemma}
\newtheorem{Prop}[Def]{Proposition}
\newtheorem{Thm}[Def]{Theorem}
\newtheorem{Cor}[Def]{Corollary}

\newtheorem{Rmk}[Def]{Remark}
\newtheorem{Alg}[Def]{Algorithm}

\makeatletter
\def\bar{\accentset{{\cc@style\underline{\mskip13mu}}}}

\makeatother

\title{{\bf\textsf{Generalized Howe curves of genus 4, 5, and 6\\ with completely decomposable Jacobians}}}
\author{Ryo Ohashi}

\begin{document}
\maketitle
\begin{abstract}
    Superspecial curves are important objects in number theory and algebraic geometry, and the existence of such a curve of genus $g \geq 4$ is an open problem, for general characteristics $p > 0$.
    As a computational approach to this problem, Kudo-Harashita-Howe~(2020) showed that a superspecial curve of genus 4 exists in each characteristic $p$ with $7 < p < 20000$.
    Their method restricted attention to a specific class of curves, known as \emph{Howe curves}, for which superspeciality is reduced to those of curves of genus at most 2.
    In this paper, we focus on a more specific class of curves, namely Howe curves whose Jacobians decompose into a product of four elliptic curves.
    By restricting our attention to such curves, the superspeciality reduces to the supersingularity of elliptic curves, which enables us to construct a superspecial curve of genus 4 more efficiently than Kudo-Harashita-Howe's method.
    As our first main result, we confirmed by computer the existence of such superspecial curves of genus 4 in characteristics $p$ with $20000 < p < 10^6$.
    Using a similar approach, we also propose constructions of superspecial curves of genera 5 and 6 from only supersingular elliptic curves.
    Furthermore, computational experiments establish the existence of superspecial curves of genus 5 (resp. genus 6) in characteristics $p$ with $13 < p < 10^5\hspace{-0.3mm}$ (resp. $7 < p < 10^5$).
\end{abstract}

\section{Introduction}
Throughout this paper, a \emph{curve} always means a non-singular projective variety of dimension one defined over a field of characteristic $p > 0$.
A curve is said to be \emph{superspecial} if its Jacobian is isomorphic to a product of supersingular elliptic curves (as unpolarized abelian varieties).
Superspecial curves are important objects not only in number theory but also in applications, including isogeny-based cryptography and algebraic geometry codes.
It is very natural to ask the following question:\medskip

\emph{Question.} For given $g$ and $p$, does there exist a superspecial curve of genus $g$ in characteristic $p$?\vspace{3mm}

\noindent This question is completely solved for genera $g \leq 3$.
In fact, Deuring~\cite{Deuring} showed that a supersingular elliptic curve (i.e., the case of $g=1$) exists in any characteristic $p > 0$.
Also, the non-existence of superspecial curves of genus $g=2$ for $p \in\hspace{-0.3mm} \{2,3\}$, and the existence of such curves for $p > 3$ were shown by Ibukiyama, Katsura, and Oort in \cite[Theorem 3.3]{IKO}.
Ekedahl proved in \cite[Theorem 1.1]{Ekedahl} that if there exists a superspecial curve of genus $g$ in characteristic $p$, then the inequality $2g \leq p^2-p$ holds, which implies that no superspecial curve of\\ genus $g=3$ exists in characteristic $p=2$.
The existence of superspecial curves of genus $g=3$ for $p > 2$ was shown by Oort~\cite[Theorem 5.12]{Oort} or Ibukiyama~\cite{Ibukiyama}.
On the other hand, the question for genus $g \geq 4$ remains open in general $p$.
One of the difficulties in this case lies in the fact that abelian varieties of dimension $g \geq 4$ are generally  not Jacobians, in contrast to the case $\hspace{-0.2mm}g \leq 3$.

In this paper, we restrict our attention to the cases $g \in\hspace{-0.3mm} \{4,5,6\}$.
We first review some previous results on the existence of superspecial curves of genus $g=4$ in small characteristic $p$.
By Ekedahl's bound, there is no\\ superspecial curve of genus 4 for $p \in\hspace{-0.2mm} \{2,3\}$.
The existence of a superspecial curve of genus $4$ for $p=5$ follows from the result \cite[Theorem 3.1]{FGT} by Fuhrmann, Garcia, and Torres, whereas the non-existence for $p=7$ was shown by Kudo and Harashita~\cite{KH}.
In addition, Kudo, Harashita, and Howe~\cite{KHH} verified that a superspecial curve of genus $4$ exists for every prime $p$ satisfying $p \equiv 5 \pmod{6}$ or $7 < p < 20000$.
Their strategy was to focus on a specific class of genus-$4$ curves called \hspace{-0.2mm}\emph{Howe curves}, which enables us to reduce the superspeciality of such curves  to that of curves of genus at most $2$; \hspace{0.2mm}we refer the reader to Section \ref{subsec:generalized} for their constructions.
To verify the existence of superspecial curves of genus $4$ in characteristics greater than $20000$, we restrict our search to a subclass of Howe curves.
Explicitly, for $s,t \notin\hspace{-0.1mm} \{0,\pm1\}$, let $E_1$ and $E_2$ be elliptic curves defined by\vspace{-1.2mm}
\begin{align*}
    E_1: y^2 &= (x-1)(x-s)(x-t),\\[-1.5mm]
    E_2: y^2 &= (x+1)(x+s)(x+t),\\[-5.6mm]
\end{align*}
and then we consider the genus-4 curve $X_{s,t}$ obtained as the desingularization of the fiber product $E_1 \hspace{-0.3mm}\times_{\mathbb{P}^1}\hspace{-0.4mm} E_2$.
Since the superspeciality of $X_{s,t}$ is reduced to the supersingularity of three elliptic curves (Proposition \ref{prop:genus4_decomposition}), this allows a more efficient construction of superspecial curves of genus $4$.
Our main result for genus-$4$ curves, obtained by a computational method, is as follows:\vspace{-1.2mm}
\begin{Thm}\label{thm:genus4}
For every prime $p$ with $7 < p < 10^6$ and $p \neq 13, 19, 73$, there exists a superspecial curve $X_{s,t}$ of genus $4$ in characteristic $p$.\vspace{-0.3mm}
\end{Thm}

\noindent By combining Theorem~\ref{thm:genus4} with Kudo-Harashita-Howe's result, the existence of superspecial genus-4 curves for all $p$ with $7 < p < 10^6$ follows.

We now turn to the case of genera $g \in\hspace{-0.3mm} \{5,6\}$; \hspace{0.3mm}relatively little is known about the existence of superspecial curves of genus $5$ and $6$.
By Ekedahl's bound, there are no superspecial curves of genus $5$ and $6$ for $p \in \{2,3\}$.
The non-existence holds also for $p = 5$.
Indeed, if there were a superspecial curve of genus 5 (resp. genus 6) in characteristic $p = 5$, then \cite[Lemma 2.2.1]{KH} would imply that a maximal curve of genus 5 (resp. genus 6) exists over $\mathbb{F}_{5^2}\hspace{-0.3mm}$.
This contradicts a result of Fuhrmann and Torres~\cite{FT}, which asserts that if a maximal curve of genus $g$ exists over $\mathbb{F}_{p^2}\hspace{-0.2mm}$, then either $4g \leq (p-1)^2$ or $2g = (p-1)p$.
To the best of the author's knowledge, no results are known that show the non-existence of superspecial curves of genus $5$ and $6$ for $p \geq 7$.
We note that a superspecial genus-5 curve exists if $p$ satisfies certain congruence conditions (cf. \cite{Kudo},\hspace{0.7mm}\cite{OK2}). Although unpublished, Kudo, Moriya, and the author~\cite{KMO} reported that a superspecial curve of genus $5$ exists for all $p$\\ with $7 \leq p < 100$ and $p \ne 13$, constructed from \hspace{-0.2mm}\emph{generalized Howe curves} defined in \cite{KT24}.
Their method is not very efficient, since it requires generating a large number of superspecial curves of genus 2.
In this paper, we therefore restrict our search to a subclass of generalized Howe curves of genera $5$ and $6$, whose \hspace{-0.1mm}Jacobians are completely decomposable; \hspace{0.3mm}see Sections \ref{sec:genus5} and \ref{sec:genus6} for explicit models of these curves.
As in the case of genus $4$, the superspeciality of our curves can be reduced to the supersingularity of three elliptic curves, which enables an efficient construction of superspecial curves of genus $5$ and $6$. 
Here are our main results for genus-5 curves and genus-6 curves:\vspace{-0.9mm}
\begin{Thm}\label{thm:genus5}
For every prime $p$ with $7 \leq p < 10^5$ and $p \neq 13$, there exists a superspecial curve of genus $5$ in characteristic $p$.\vspace{-1.3mm}
\end{Thm}
\begin{Thm}\label{thm:genus6}
For every prime $p$ with $7 < p < 10^5$, there exists a superspecial curve of genus $6$ in charact-\\eristic $p$.\vspace{-0.2mm}
\end{Thm}
\noindent For details of the computational experiments, see Sections~\ref{subsec:experiment_for_genus5} and \ref{subsec:experiment_for_genus6}.
The above results can be summarized in Table \ref{tbl:summary}, where \textcolor{cyan}{$\raisebox{0.3mm}{\scalebox{0.75}{$\checkmark$}}$} (resp. \textcolor{magenta}{$\hspace{-0.3mm}\times$}) indicates existence (resp. non-existence) of superspecial curves:\vspace{-2.8mm}

\begin{table}[h]
    \centering
    \caption{The existence of superspecial curves of genus $g$ in characteristic $p$}\label{tbl:summary}\vspace{2mm}
    \begin{tabular}{c|c|c|c|c|c|c|l}
        $\!\!g \hspace{0.1mm}\backslash\hspace{0.5mm} p\!$ & $2$ & $3$ & $5$ & $7$ & $11$ & $13$ & $\hspace{9.5mm}\cdots$\\\hline
        $1$ & \textcolor{cyan}{$\scalebox{0.75}{$\checkmark$}$} & \textcolor{cyan}{$\scalebox{0.75}{$\checkmark$}$} & \textcolor{cyan}{$\scalebox{0.75}{$\checkmark$}$} & \textcolor{cyan}{$\scalebox{0.75}{$\checkmark$}$} & \textcolor{cyan}{$\scalebox{0.75}{$\checkmark$}$} & \textcolor{cyan}{$\scalebox{0.75}{$\checkmark$}$} & \textcolor{cyan}{$\scalebox{0.75}{$\checkmark$}$} \text{\,for all} $p$\\\hline
        $2$ & \textcolor{magenta}{$\times$} & \textcolor{magenta}{$\times$} & \textcolor{cyan}{$\scalebox{0.75}{$\checkmark$}$} & \textcolor{cyan}{$\scalebox{0.75}{$\checkmark$}$} & \textcolor{cyan}{$\scalebox{0.75}{$\checkmark$}$} & \textcolor{cyan}{$\scalebox{0.75}{$\checkmark$}$} & \textcolor{cyan}{$\scalebox{0.75}{$\checkmark$}$} \text{\,for all} $p$\\\hline
        $3$ & \textcolor{magenta}{$\times$} & \textcolor{cyan}{$\scalebox{0.75}{$\checkmark$}$} & \textcolor{cyan}{$\scalebox{0.75}{$\checkmark$}$} & \textcolor{cyan}{$\scalebox{0.75}{$\checkmark$}$} & \textcolor{cyan}{$\scalebox{0.75}{$\checkmark$}$} & \textcolor{cyan}{$\scalebox{0.75}{$\checkmark$}$} & \textcolor{cyan}{$\scalebox{0.75}{$\checkmark$}$} \text{\,for all} $p$\\\hline
        $4$ & \textcolor{magenta}{$\times$} & \textcolor{magenta}{$\times$} & \textcolor{cyan}{$\scalebox{0.75}{$\checkmark$}$} & \textcolor{magenta}{$\times$} & \textcolor{cyan}{$\scalebox{0.75}{$\checkmark$}$} & \textcolor{cyan}{$\scalebox{0.75}{$\checkmark$}$} & \textcolor{cyan}{$\scalebox{0.75}{$\checkmark$}$} \text{\,for\,} $p < 10^6$\\\hline
        $5$ & \textcolor{magenta}{$\times$} & \textcolor{magenta}{$\times$} & \textcolor{magenta}{$\times$} & \textcolor{cyan}{$\scalebox{0.75}{$\checkmark$}$} & \textcolor{cyan}{$\scalebox{0.75}{$\checkmark$}$} & ?? & \textcolor{cyan}{$\scalebox{0.75}{$\checkmark$}$} \text{\,for\,} $p < 10^5$\\\hline
        $6$ & \textcolor{magenta}{$\times$} & \textcolor{magenta}{$\times$} & \textcolor{magenta}{$\times$} & ?? & \textcolor{cyan}{$\scalebox{0.75}{$\checkmark$}$} & \textcolor{cyan}{$\scalebox{0.75}{$\checkmark$}$} & \textcolor{cyan}{$\scalebox{0.75}{$\checkmark$}$} \text{\,for\,} $p < 10^5$
    \end{tabular}
\end{table}\vspace{-0.3mm}

\noindent It remains open whether superspecial curves exist for $(g,p) = (5,13)$ and $(6,7)$.

The remainder of this paper is organized as follows.
Section \ref{sec:preliminaries} is devoted to preliminaries for the subsequent sections.
Explicitly, Section \ref{subsec:decomposition} summarizes basic facts on the Jacobian decomposition of genus-$2$ curves, and Section \ref{subsec:generalized} recalls the definition of (generalized) Howe curves.
In Section \ref{sec:genus4}, we construct a family of genus-4 curves whose Jacobians are completely decomposable and propose our algorithm for generating a superspecial curve of genus 4.
Similarly, in Section \ref{sec:genus5} \hspace{0.3mm}(resp. Section \ref{sec:genus6}), we explain our algorithm for verifying the existence of superspecial curves of genus 5 (resp. genus 6).
Finally, Section \ref{sec:experiment} presents the results obtained by executing these algorithms, along with the proofs of Theorems~\ref{thm:genus4}, \ref{thm:genus5}, and \ref{thm:genus6}.

\section{Preliminaries}\label{sec:preliminaries}
In this section, we recall several results on the Jacobian decomposition of genus-2 curves, together with the definition of generalized Howe curves.
Let $K$ be a field of characteristic $p > 5$ throughout this paper.\vspace{-1mm}

\subsection{Jacobian decomposition of genus-2 curves}\label{subsec:decomposition}
For a genus-2 curve $C$ over $K$, we denote by ${\rm Aut}(C)$ the group of all automorphisms of $C$ over $\bar{K}$, and define the quotient\vspace{-1.4mm}
\[
    \mathrm{RA}(C) \coloneqq \mathrm{Aut}(C)/\langle\iota_C\rangle\vspace{0.8mm}
\]
to be the \emph{reduced automorphism group} of $C$, where $\iota_C$ denotes the hyperelliptic involution of $C$.
The possible reduced automorphism groups of genus-2 curves were classified by Igusa~\cite[Section 8]{Igusa} into the following seven finite groups:\vspace{-0.9mm}
\[
    \{1\},\ \mathbb{Z}/2\mathbb{Z},\ S_3,\ D_4 = \mathbb{Z}/2\mathbb{Z} \times \mathbb{Z}/2\mathbb{Z},\ D_{12},\ S_4,\ \mathbb{Z}/5\mathbb{Z},\vspace{0.7mm}
\]
where $D_{2n}$ and $S_n$ denote the dihedral group of order $2n$ and the symmetric group of degree $n$, respectively.
Table~\ref{g2-auto} provides a list of defining equations for genus-2 curves corresponding to each reduced automorphism group, taken from \cite[Section 5]{KT20}.\vspace{-3.5mm}
\begin{table}[htbp]
    \centering
    \caption{Classification of genus-2 curves $C$ in characteristic $p > 5$}\label{g2-auto}\vspace{2mm}
    \begin{tabular}{c||c|l}
        Case & $\mathrm{RA}(C)$ & A  defining equation of $C$\\\hline
        0 & $\{1\}$ & $y^2 = (\text{square-free polynomial in $x$ of degree $5$ or $6$})$\\\hline
        1 & $\mathbb{Z}/2\mathbb{Z}$ & $y^2 = (x^2-1)(x^2-a^2)(x^2-b^2)$\\\hline
        2 & $S_3$ & $y^2 = (x^3-1)(x^3-a^3)$\\\hline
        3 & $D_4$ & $y^2 = x(x^2-1)(x^2-a^2)$\\\hline
        4 & $D_{12}$ & $y^2 = x^6-1$\\\hline
        5 & $S_4$ & $y^2 = x(x^4-1)$\\\hline
        6 & $\mathbb{Z}/5\mathbb{Z}$ & $y^2 = x^5-1$
    \end{tabular}
\end{table}\vspace{-0.9mm}

\noindent Here, it should be noted that:\vspace{-1.1mm}
\begin{itemize}
    \item Each curve in Case 2 or 3 is a specialization of a curve in Case 1, and\vspace{-2.6mm}
    \item Each curve in Case 4 or 5 is a specialization of curves in both Case 2 and Case 3.\vspace{-0.3mm}
\end{itemize}
We refer to \cite[Theorem 1.1]{Ishii} for explicit conditions on the parameters $a$ and $b$ in Table~\ref{g2-auto} which yield these specializations.

The Jacobian of a curve in Case 1 (or of its specialization) is isogenous to a product of two elliptic curves. 
More precisely, for a genus-2 curve\vspace{-1.7mm}
\begin{equation}\label{eq:Case1}
    C: y^2 = (x^2-1)(x^2-a^2)(x^2-b^2)\vspace{-0.6mm}
\end{equation}
where $a,b \in \bar{K}\!\smallsetminus\!\{0,\pm1\}$ with $a^2 \hspace{-0.3mm}\neq b^2$, the involutions $\sigma: (x,y) \mapsto (-x,y)$ and $\tau: (x,y) \mapsto (-x,-y)$ define the quotient maps $C \to C/\langle\sigma\rangle \eqqcolon E_\sigma$ and $C \to C/\langle\tau\rangle \eqqcolon E_\tau$, where $E_\sigma$ and $E_\tau$ are elliptic curves given by\vspace{-0.4mm}
\begin{alignat*}{2}
    E_\sigma: v^2 &= (u-1)(u-a^2)(u-b^2) && \,\text{ with }\, u = x^2,\ v = y,\\[-1mm]
    E_\tau: v^2 &= u(u-1)(u-a^2)(u-b^2) && \,\text{ with }\, u = x^2,\ v = xy.\\[-4.4mm]
\end{alignat*}
These projections $C \rightarrow E_\sigma$ and $C \rightarrow E_\tau$ induce the isogeny $\hspace{-0.1mm}J(C) \to E_\sigma \times E_\tau$.
By transforming $E_\sigma,E_\tau$ into Legendre form, we obtain the following proposition:\vspace{-1mm}
\begin{Prop}[cf. {\cite[Section 2]{KO}}]\label{prop:Case1}
Let $C$ be the genus-2 curve defined by the equation~\eqref{eq:Case1}.
Then, the Jacobian of $C$ is isogenous to $E_\sigma \times E_\tau$, where $E_\sigma$ and $E_\tau$ are the elliptic curves\vspace{-0.1mm}
\begin{align*}
    E_\sigma: v^2 &= u(u-1)\biggl(u-\frac{b^2-a^2}{1-a^2}\biggr),\\[-0.3mm]
    E_\tau: v^2 &= v(u-1)\biggl(u-\frac{b^2-a^2}{b^2(1-a^2)}\biggr),
\end{align*}
and the degree of the isogeny is a power of $2$.
\end{Prop}

\newpage
In particular, the Jacobian of a genus-2 curve in Cases 2--5 is isogenous to the square of an elliptic curve.
Explicitly, the following propositions are known:\vspace{-1.2mm}
\begin{Prop}[cf. {\cite[Proposition 4.1]{FS}}]\label{prop:Case3}
Let $C: y^2 = x(x^2-1)(x^2-a^2)$ be a genus-2 curve with $a^2 \neq 0,1$.
Then, the Jacobian of $C$ is isogenous to $E^2$, where $E$ is the elliptic curve\vspace{-1.5mm}
\[
    E: v^2 = (c+2)u^3 - (3c-10)u^2 + (3c-10)u - (c+2)\vspace{-1.8mm}
\]
with the $j$-invariant\vspace{-1.4mm}
\[
    j(E) = 64\frac{(3c-10)^3}{(c-2)(c+2)^2},\vspace{-0.9mm}
\]
for $c \in \bar{K}$ satisfying $c^2 = a^2 + 1/a^2 + 2$, and the degree of the isogeny is a power of 2.\vspace{-1.3mm}
\end{Prop}
\begin{Prop}[cf. {\cite[Proposition 4.2]{FS}}]\label{prop:Case2}
Let $C: y^2 = (x^3-1)(x^3-a^3)$ be a genus-2 curve with $a^3 \neq 0,1$.
Then, the Jacobian of $C$ is isogenous to $E^2$, where $E$ is the elliptic curve\vspace{-1.5mm}
\[
    E: v^2 = (c+2)u^3 - (3c-30)u^2 + (3c+30)u - (c-2)\vspace{-1.8mm}
\]
with the $j$-invariant\vspace{-1.4mm}
\[
    j(E) = 6912\frac{(2c-5)^3}{(c-2)(c+2)^3},\vspace{-0.9mm}
\]
for $c \in \bar{K}$ satisfying $c^2 = a^3 + 1/a^3 + 2$, and the degree of the isogeny is a product of powers of $2$ and $3$.\vspace{-1.3mm}
\end{Prop}
\begin{Rmk}
Since the curve in Case 4 corresponds to the case $a^3 = -1$ in Proposition \ref{prop:Case2}, its Jacobian is isogenous to the square of the elliptic curve with $j$-invariant $54000$, which has CM by $\mathbb{Z}[\hspace{-0.1mm}\sqrt{-3}]$.
Similarly, the curve in Case 5 corresponds to the case $a^2 = -1$ in Proposition \ref{prop:Case3}, and thus its Jacobian is isogenous to the square of the elliptic curve with $j$-invariant $8000$, which has CM by $\mathbb{Z}[\hspace{-0.1mm}\sqrt{-2}]$.
\end{Rmk}

\subsection{Generalized Howe curves}\label{subsec:generalized}
Let $C_1$ and $C_2$ be hyperelliptic curves of genera $g_1$ and $g_2$ which share exactly $r$ common \hspace{-0.3mm}Weierstrass points; that is,\vspace{-3.9mm}
\begin{align*}
    C_1: y_1^2 &= (x-a_1)(x-a_2) \cdots (x-a_r)(x-b_{r+1}) \cdots (x-b_{2g_1+2}),\\[-1.3mm]
    C_2: y_2^2 &= (x-a_1)(x-a_2) \cdots (x-a_r)(x-c_{r+1}) \cdots (x-c_{2g_2+2}),\\[-6,5mm]
\end{align*}
where $a_i,b_i,c_i \in \mathbb{P}^1(\bar{K})$ are all distinct.
The desingularization $H$ of the fiber product $C_1 \times_{\mathbb{P}^1} C_2$ is referred to\\ as a \emph{generalized Howe curve}.
We recall the results by Katsura and Takashima \cite[Section 2]{KT24}
on the genus of a generalized Howe curve $H$ and a criterion for $H$ to be 
hyperelliptic:\vspace{-1.6mm}
\begin{Prop}\label{prop:genus}
With the above notation, the genus of $H$ is given by $2g_1+2g_2+1-r$.
Moreover, if the genus of $H$ is at least $4$, then $H$ is hyperelliptic if and only if $r = g_1+g_2+1$.
\end{Prop}

Next, the hyperelliptic involutions of $C_1$ and $C_2$ induce the involutions of $H$ given by\vspace{-1.8mm}
\begin{align*}
    \sigma_1: (x,y_1,y_2) &\longmapsto (x,y_1,-y_2),\\[-1.1mm]
    \sigma_2: (x,y_1,y_2) &\longmapsto (x,-y_1,y_2),
\end{align*}
and we clearly have the isomorphisms $H/\langle\sigma_i\rangle \cong C_i$ for each $i \in \{1,2\}$.
The composition $\sigma_1\sigma_2$ defines another quotient map $H \rightarrow H/\langle \sigma_1\sigma_2\rangle \eqqcolon C_3$, where $C_3$ is the hyperelliptic curve given by\vspace{-1.2mm}
\[
    C_3: y_3^2 = (x-b_{r+1}) \cdots (x-b_{2g_1+2})(x-c_{r+1}) \cdots (x-c_{2g_2+2})
\]
with $y_3 = y_1y_2/(x-a_1)(x-a_2) \cdots (x-a_r)$.
We remark that the genus of $C_3$ is given by $g_1+g_2+1-r$, and thus the genus of $H$ equals the sum of the genera of $C_1,C_2$, and $C_3$.
An important property of a generalized Howe curve $H$ is that their Jacobians are isogenous to the product of the Jacobians of $C_1,C_2$, and $C_3$.\vspace{-1.4mm}
\begin{Thm}[cf. {\cite[Theorem 3]{KT24}}]\label{thm:decomposition}
With the above notation, there exists an isogeny\vspace{-1.4mm}
\[
    J(H) \longrightarrow J(C_1) \times J(C_2) \times J(C_3)\vspace{-0.7mm}
\]
whose degree is a power of 2.
In particular, the curve $H$ is superspecial if and only if $C_1,C_2$, and $C_3$ are all superspecial.
\end{Thm}

\newpage
Now assume that $g_1=g_2=1$ and $r=1$.
In other words, let $E_1$ and $E_2$ be elliptic curves defined by\vspace{-0.7mm}
\begin{align}
    E_1: y^2 &= (x-a)(x-b_2)(x-b_3)(x-b_4),\label{eq:E1}\\[-1.5mm]
    E_2: y^2 &= (x-a)(x-c_2)(x-c_3)(x-c_4),\label{eq:E2}\\[-6.2mm]\nonumber
\end{align}
where $a,b_2,b_3,b_4,c_2,c_3,c_4 \in \mathbb{P}^1(\bar{K})$ are pairwise distinct, and define $H$ as the desingularization of $E_1 \times_{\mathbb{P}^1}\hspace{-0.2mm} E_2$.
Proposition~\ref{prop:genus} shows that $H$ has genus 4 and is non-hyperelliptic; \,such a generalized Howe curve $H$ is simply called a \emph{Howe curve}.
By Theorem~\ref{thm:decomposition}, a Howe curve $H$ is superspecial if and only if two elliptic curves $E_1,E_2$ are both supersingular and the genus-2 curve\vspace{-1.1mm}
\begin{equation}\label{eq:C3}
    C_3: y^2 = (x-b_2)(x-b_3)(x-b_4)(x-c_2)(x-c_3)(x-c_4)
\end{equation}
is superspecial.
Based on this property, Kudo, Harashita, and Howe~\cite{KHH} proposed an algorithm for generating superspecial Howe curves.
For completeness, we briefly outline their algorithm below:\vspace{-0.7mm}
\begin{enumerate}
\setlength{\leftskip}{22pt}
\item[{\it Step 1.}\,] \hspace{-0.2mm}Choose parameters $(b_2,b_3,b_4,c_2,c_3,c_4)$ such that the genus-2 curve \eqref{eq:C3} is superspecial.\vspace{-1.2mm}
\item[{\it Step 2.}\,] For each $a$ for which the elliptic curve \eqref{eq:E1} is supersingular, check whether the elliptic curve \eqref{eq:E2} is also supersingular.
If so, output $(a,b_2,b_3,b_4,c_2,c_3,c_4)$.\vspace{-1mm}
\item[{\it Step 3.}\,] If no such parameter $a$ is found, return to {\it Step 1} with another choice of $(b_2,b_3,b_4,c_2,c_3,c_4)$.
\end{enumerate}
Each output $(a,b_2,b_3,b_4,c_2,c_3,c_4)$ determines a superspecial Howe curve $H$.
Applying the above algorithm, they\hspace{-0.2mm} verified the existence of superspecial Howe curves for every characteristic $p$ with $7 < p < 20000$.
Their implementation is available at the URL cited in \cite[Section 6]{KHH}.

\vspace{-1mm}
\begin{Rmk}
The number of superspecial Howe curves in characteristic $p$ is expected to be $\varTheta(p^3)$.
Indeed, since the number of superspecial genus-2 curves
is approximately $\hspace{-0.1mm}p^3/2880$ as shown in \cite[Theorem 3.3]{IKO}, there are $\varTheta(p^3)$ possible choices of parameters $(b_2,b_3,b_4,c_2,c_3,c_4)$ in {\it Step 1}.
For a fixed $(b_2,b_3,b_4,c_2,c_3,c_4)$, there are approximately $p/2$ values of $a$ for which the elliptic curve \eqref{eq:E1} is supersingular, and the probability that the elliptic curve \eqref{eq:E2} is also supersingular is estimated to be about $1/(2p)$.
This heuristic explains why this algorithm is expected to produce many superspecial Howe curves for sufficiently large $p$.
\end{Rmk}

\section{Genus-4 curves with completely decomposable Jacobians}\label{sec:genus4}
\setcounter{equation}{0}
In this section, we construct a family of genus-4 curves whose Jacobians are completely decomposable, namely, isogenous to a product of elliptic curves.
More precisely, we restrict our attention to Howe curves such that the genus-2 curve \eqref{eq:C3} falls into Case 1 (or its specialization).
The main purpose of this section is to present an algorithm (Algorithm \ref{alg:genus4}) that generates superspecial curves among them.

For parameters $s,t \in \bar{K}\!\smallsetminus\! \{0,\pm1\}$ with $s^2 \neq t^2$, we consider two elliptic curves\vspace{-1.3mm}
\begin{align*}
    E_1: y^2 &= (x-1)(x-s)(x-t),\\[-1.4mm]
    E_2: y^2 &= (x+1)(x+s)(x+t).
\end{align*}
We denote by $X_{s,t}$ the desingularization of the fiber product $E_1 \times_{\mathbb{P}^1}\hspace{-0.4mm} E_2$, which is a Howe curve as discussed in Section \ref{subsec:generalized}.
In particular, the curve $X_{s,t}$ has genus 4 and is non-hyperelliptic.
Moreover, by Theorem~\ref{thm:decomposition}, the Jacobian of $X_{s,t}$ is isogenous to $E_1 \times E_2 \times J(C_3)$, where $C_3$ is the genus-2 curve defined by\vspace{-0.8mm}
\[
    C_3: y^2 = (x^2-1)(x^2-s^2)(x^2-t^2).
\]
Then, it follows from Proposition~\ref{prop:Case1} that the Jacobian of $C_3$ is isogenous to $E_3 \times E_4$, where\vspace{-0.6mm}
\begin{align}
    E_3: y^2 &= x(x-1)(x-\lambda_3), \quad \lambda_3 \coloneqq \frac{t^2-s^2}{1-s^2},\nonumber\\[-0.3mm]
    E_4: y^2 &= x(x-1)(x-\lambda_4), \quad \lambda_4 \coloneqq \frac{t^2-s^2}{t^2(1-s^2)},\label{eq:E4}\\[-6.1mm]\nonumber
\end{align}
and therefore the Jacobian of $X_{s,t}$ is completely decomposable.
Since $E_2$ is isomorphic to $E_1$, we obtain the following proposition:\vspace{-0.8mm}
\begin{Prop}\label{prop:genus4_decomposition}
With the above notation, there exists an isogeny $J(X_{s,t}) \rightarrow E_1^2 \times E_3 \times E_4$ whose degree is a power of $2$.
In particular, the curve $X_{s,t}$ is superspecial if and only if $E_1,E_3$, and $E_4$ are all supersingular.
\end{Prop}

\newpage
As a corollary, if the characteristic $p$ satisfies $p \equiv 5 \pmod{6}$, then there exists a superspecial 
$X_{s,t}$, which coincides with the one constructed by 
Kudo-Harashita-Howe~\cite[Theorem 1.1]{KHH}.\vspace{-0.8mm}
\begin{Cor}\label{cor:existence_genus4}
If $p \equiv 5 \pmod{6}$, then a superspecial curve $X_{s,t}$ of genus $4$ in characteristic $p$ exists.\vspace{-0.5mm}
\end{Cor}
\begin{proof}
We define $s \coloneqq \omega$ and $t \coloneqq \omega^2$, where $\omega$ denotes a primitive cube root of unity.
Then, the three elliptic curves $E_1,E_3$, and $E_4$ are given by\vspace{-0.8mm}
\begin{align*}
    E_1: y^2 &= (x-1)(x-\omega)(x-\omega^2),\\[-1.1mm]
    E_3: y^2 &= x(x-1)(x-1-\omega),\\[-1.1mm]
    E_4: y^2 &= x(x-1)(x-1-\omega^2).
\end{align*}
The $j$-invariants of these curves are all equal to $0$, and hence they are supersingular (cf. \cite[Example \hspace{-0.3mm}V.4.4]{Silverman}).
Therefore, it follows from Proposition~\ref{prop:genus4_decomposition} that the curve $X_{s,t}$ is superspecial.
\end{proof}

Next, the elliptic curve $E_1$ can be written in Legendre form as\vspace{-1.3mm}
\[
    E_1: y^2 = x(x-1)(x-\lambda_1), \quad \lambda_1 \coloneqq \frac{t-s}{1-s}\vspace{-0.6mm}
\]
by the Möbius transformation $x \mapsto (x-s)/(1-s)$.
Note that the non-singularity of $E_1,E_3$, and $E_4$ implies that $\lambda_1,\lambda_3,\lambda_4 \neq 0,1$.
Additionally, the parameters $\lambda_1$ and $\lambda_3$ satisfy the following conditions:\vspace{-0.5mm}
\begin{equation}\label{eq:lambda}
    \left\{\begin{array}{l}
        \lambda_1-\lambda_3 \neq 0,\\
        \lambda_1^2-\lambda_3 \neq 0,\\
        \lambda_1^2-2\lambda_1+\lambda_3 \neq 0,\\
        2\lambda_1\lambda_3 - \lambda_1^2-\lambda_3 \neq 0.
    \end{array}\right.\vspace{-0.2mm}
\end{equation}
Indeed, a straightforward computation yields\vspace{-1.4mm}
\begin{align*}
    \lambda_1-\lambda_3 &= \frac{(1-t)(t-s)}{(1-s)(1+s)},\nonumber\\
    \lambda_1^2-\lambda_3 &= -\frac{2s(1-t)(t-s)}{(1-s)^2(1+s)},\\
    \lambda_1^2-2\lambda_1+\lambda_3 &= -\frac{2(1-t)(t-s)}{(1-s)^2(1+s)},\\
    2\lambda_1\lambda_3 - \lambda_1^2-\lambda_3 &= -\frac{2t(1-t)(t-s)}{(1-s)^2(1+s)},\\[-5.8mm]
\end{align*}
which are all non-zero since $s,t \notin \{0,\pm1\}$ and $s^2 \neq t^2$.
By combining suitable expressions among the above, we can express $s$ and $t$ in terms of $\lambda_1$ and $\lambda_3$ explicitly:\vspace{-0.5mm}
\begin{equation}\label{eq:st_for_genus4}
    s = \frac{\lambda_1^2-\lambda_3}{\lambda_1^2-2\lambda_1+\lambda_3}, \quad t = \frac{2\lambda_1\lambda_3 - \lambda_1^2-\lambda_3}{\lambda_1^2-2\lambda_1+\lambda_3}.\vspace{-0.2mm}
\end{equation}
We now determine the constraints on $\lambda_1$ and $\lambda_3$ under which our assumptions imposed to ensure that $C_3$ is non-singular, namely $s,t \notin \{0,\pm1\}$ and $s^2 \neq t^2$, are satisfied.\vspace{-1mm}
\begin{Lem}
For any $\lambda_1,\lambda_3 \in \bar{K}\!\smallsetminus\!\{0,1\}$ satisfying the conditions \eqref{eq:lambda}, the corresponding values of $s$ and $t$ given by \eqref{eq:st_for_genus4} lie in $\bar{K}\!\smallsetminus\!\{0,\pm1\}$ with $s^2 \neq t^2$.\vspace{-0.6mm}
\end{Lem}
\begin{proof}
By \eqref{eq:lambda}, both the numerators and the denominator 
in \eqref{eq:st_for_genus4} are non-zero; hence $s$ and $t$ are well-defined 
and non-zero.
Also, a straightforward computation yields\vspace{-0.6mm}
\begin{alignat*}{3}
    s-1 &=\frac{2(\lambda_1-\lambda_3)}{\lambda_1^2-2\lambda_1+\lambda_3},& \quad
    s+1 &=\frac{2\lambda_1(\lambda_1-1)}{\lambda_1^2-2\lambda_1+\lambda_3},\\
    t-1 &=-\frac{2(\lambda_1-1)(\lambda_1-\lambda_3)}{\lambda_1^2-2\lambda_1+\lambda_3},& \quad
    t+1 &=\frac{2\lambda_1(\lambda_3-1)}{\lambda_1^2-2\lambda_1+\lambda_3},\\
    s-t &=\frac{2\lambda_1(\lambda_1-\lambda_3)}{\lambda_1^2-2\lambda_1+\lambda_3},& \quad
    s+t &=\frac{2\lambda_3(\lambda_1-1)}{\lambda_1^2-2\lambda_1+\lambda_3},\\[-5.7mm]
\end{alignat*}
and each of these values is non-zero.
This implies that $s \neq \pm 1$, $t \neq \pm 1$, and $s^2 \neq t^2$, as desired.
\end{proof}

\newpage
From the above discussion, for given $\lambda_1,\lambda_3 \in \bar{K}\!\smallsetminus\!\{0,1\}$ satisfying the conditions \eqref{eq:lambda}, our curve $X_{s,t}$ can be constructed via \eqref{eq:st_for_genus4}.
In particular, by choosing $\lambda_1$ and $\lambda_3$ so that $E_1$ and $E_3$ are supersingular, it follows from Proposition~\ref{prop:genus4_decomposition} that $X_{s,t}$ is superspecial if and only if $E_4$ is also supersingular.
Hence, we can produce superspecial $X_{s,t}$ as follows:\vspace{-1.1mm}
\begin{Alg}\label{alg:genus4}
\hspace{1mm}Input: A prime $p > 5$.\\
\hspace{25.7mm}Output: A pair $(s,t)$ such that $X_{s,t}$ is superspecial in characteristic $p$, or $\bot$.\vspace{-0.9mm}
\begin{enumerate}
\setlength{\leftskip}{22pt}
\item[{\it Step 1.}\,] \hspace{-0.2mm}Compute the set $\mathcal{T}$ of values $\lambda \neq 0,1$ for which the elliptic curve $y^2 = x(x-1)(x-\lambda)$ in Legendre form is supersingular in characteristic $p$.\vspace{-1.2mm}
\item[{\it Step 2.}\,] For each pair $(\lambda_1,\lambda_3) \in \mathcal{T}^2\hspace{-0.2mm}$ satisfying the conditions \eqref{eq:lambda}, check whether the elliptic curve \eqref{eq:E4} is\\ supersingular, where $(s,t)$ is given by \eqref{eq:st_for_genus4}. If so, output $(s,t)$.\vspace{-1mm}
\item[{\it Step 3.}\,] If no such $(\lambda_1,\lambda_3)$ is found, output $\bot$.\vspace{-0.5mm}
\end{enumerate}
\end{Alg}
\noindent We mention that, since $\mathcal{T} \hspace{-0.2mm}\subset \mathbb{F}_{p^2}\hspace{-0.3mm}$ by the result of Auer-Top~\cite[Proposition 2.2]{AT}, the outputs $s$ and $t$ lie in $\mathbb{F}_{p^2}$.\\
Also, it is well-known (cf. \cite[Theorem \hspace{-0.4mm}V.4.1]{Silverman}) that the cardinality of $\mathcal{T}$ computed in {\it Step 1}\hspace{0.3mm} equals $(p-1)/2$.
There are $\varTheta(p^2)$ possible choices of pairs $(\lambda_1,\lambda_3)$ in {\it Step 2}, and the probability that the elliptic curve \eqref{eq:E4} is\\ supersingular is estimated to be about $1/(2p)$ for each $(\lambda_1,\lambda_3)$.
Therefore, the expected number of superspecial $X_{s,t}$ in characteristic $p$ is $\varTheta(p)$.
This suggests that, for sufficiently large $p$, Algorithm~\ref{alg:genus4} will produce at\\ least one superspecial $X_{s,t}$. In fact, within the range of our experiments, such curves are found for all $p > 73$ (see Section~\ref{subsec:experiment_for_genus4} for details).\vspace{-1.1mm}
\begin{Rmk}
In our experiments, we first run Algorithm~\ref{alg:genus4} with $\mathcal{T}$ taken to be the set of $\lambda$ corresponding to supersingular elliptic curves whose $j$-invariants lie in $\mathbb{F}_p$.
If this search fails to produce a superspecial $X_{s,t}$,\\ we then rerun Algorithm~\ref{alg:genus4} with $\mathcal{T}$ as the full set of the $\lambda$ corresponding to all supersingular elliptic curves, following the original procedure.
For most $p$ tested in our experiments, this restricted choice of $\mathcal{T}$ suffices to yield at least one superspecial $X_{s,t}$, thereby significantly reducing the execution time required for {\it Step 1}.
\end{Rmk}

\section{Genus-5 curves with completely decomposable Jacobians}\label{sec:genus5}
\setcounter{equation}{0}
In this section, we construct a family of generalized Howe curves of genus 5, specifically those whose Jacobians are completely decomposable.
Subsequently, we give an algorithm (Algorithm \ref{alg:genus5}) for generating superspecial curves among them.

Consider two genus-2 curves $C_1$ and $C_2$ sharing exactly $r=4$\hspace{-0.1mm} Weierstrass points:\vspace{-1.2mm}
\begin{align}\label{eq:generalized_genus5}
    \begin{split}
    C_1: y^2 &= (x-a_1)(x-a_2)(x-a_3)(x-a_4)(x-b_1)(x-b_2),\\[-1.5mm]
    C_2: y^2 &= (x-a_1)(x-a_2)(x-a_3)(x-a_4)(x-c_1)(x-c_2),\\[-0.9mm]
    \end{split}
\end{align}
where $a_i,b_i,c_i \in \mathbb{P}^1(\bar{K})$ are all distinct.
The desingularization of the fiber product $C_1 \times_{\mathbb{P}^1}\hspace{-0.1mm} C_2$ is a generalized Howe curve, whose genus is $5$ according to Proposition~\ref{prop:genus}.
In the following, we focus on the case where both genus-2 curves $C_1$ and $C_2$ fall into Case~3 (or its specialization).
Specifically, for parameters $s,t \in \bar{K}\!\smallsetminus\!\{0,\pm1\}$ with $s^2 \hspace{-0.2mm}\neq t^2$, let $C_1$ and $C_2$ be genus-2 curves defined by\vspace{-1.3mm}
\begin{align*}
    C_1: y^2 &= x(x^2-1)(x^2-s^2),\\[-1.4mm]
    C_2: y^2 &= x(x^2-1)(x^2-t^2).\\[-5.8mm]
\end{align*}
We denote by $Y_{s,t}$ the generalized Howe curve constructed from these curves.
By Theorem~\ref{thm:decomposition}, the Jacobian of $Y_{s,t}$ is isogenous to $J(C_1) \hspace{-0.1mm}\times\hspace{-0.1mm} J(C_2) \hspace{-0.1mm}\times\hspace{-0.1mm} E_3$, where $E_3$ is the elliptic curve defined by\vspace{-1.2mm}
\begin{equation}\label{eq:E3_for_genus5}
    E_3: y^2 = (x^2-s^2)(x^2-t^2).\vspace{-0.2mm}
\end{equation}
According to Proposition~\ref{prop:Case3}, the Jacobian of each $C_i$ is isogenous to $E_i^2$, where\vspace{-1.3mm}
\[
    E_i: y^2 = (c_i+2)x^3 - (3c_i-10)x^2 + (3c_i-10)x - (c_i+2)
\]
with $c_1 \coloneqq s + 1/s$ and $c_2 \coloneqq t + 1/t$ (we note that these values satisfy $c_1^2 = s^2 + 1/s^2 + 2$ and $c_2^2 = t^2 + 1/t^2 + 2$).
Therefore, the Jacobian of $Y_{s,t}$ is completely decomposable:\vspace{-1.2mm}
\begin{Prop}\label{prop:genus5_decomposition}
With the above notation, there exists an isogeny $\hspace{-0.2mm}J(Y_{s,t}) \rightarrow E_1^2 \times E_2^2 \times E_3$ whose degree is a power of 2.
In particular, the curve $Y_{s,t}$ is superspecial if and only if $E_1,E_2$, and $E_3$ are all supersingular.
\end{Prop}

\newpage
Here, we define the polynomial $h(\beta)$ by\vspace{-1.5mm}
\begin{equation}\label{eq:h(x)}
    h(\beta) \coloneqq \sum_{i=0}^{\lfloor p/4 \rfloor} \binom{(p-1)/2}{\lfloor(p+1)/4\rfloor+i}\hspace{-0.5mm}\binom{(p-1)/2}{i}\beta^i.\vspace{-0.9mm}
\end{equation}
According to \cite[Proposition 1.9]{IKO}, the two genus-2 curves $C_1$ and $C_2$ are superspecial if and only if $h(s^2) = 0$ and $h(t^2) = 0$, respectively.
Hence, provided that $s^2$ and $t^2$ are roots of the polynomial \eqref{eq:h(x)}, it follows from Theorem~\ref{thm:decomposition} that our curve $Y_{s,t}$ is superspecial if and only if the elliptic curve $E_3$ is supersingular.
Based on this property, one can construct a {\it naive} algorithm to generate superspecial $Y_{s,t}$ as follows:\vspace{-1.1mm}
\begin{Alg}\label{alg:pre-genus5}
\hspace{1mm}Input: A prime $p > 5$.\\
\hspace{25.7mm}Output: A pair $(s,t)$ such that $Y_{s,t}$ is superspecial in characteristic $p$, or $\bot$.\vspace{-1.1mm}
\begin{enumerate}
\setlength{\leftskip}{22pt}
\item[{\it Step 1.}\,] \hspace{-0.2mm}Compute the set $\mathcal{T}$ of roots $\beta \neq 0,1$ of the polynomial \eqref{eq:h(x)}.\vspace{-1.2mm}
\item[{\it Step 2.}\,] For each pair $(\beta_1,\beta_2) \in \mathcal{T}^2$ with $\beta_1 \hspace{-0.1mm}\neq \beta_2$, check whether the elliptic curve $y^2 = (x^2-\beta_1)(x^2-\beta_2)$ is supersingular.
If so, output any $(s,t)$ with $s^2 = \beta_1$ and $t^2 = \beta_2$.\vspace{-1.1mm}
\item[{\it Step 3.}\,] If no such $(\beta_1,\beta_2)$ is found, output $\bot$.\vspace{-0.5mm}
\end{enumerate}
\end{Alg}
\noindent We mention that, by the author's result \cite[Main Theorem A]{Ohashi}, the superspeciality of $C_1$ and $C_2$ implies that the outputs $s$ and $t$ lie in $\mathbb{F}_{p^2}\hspace{-0.2mm}$.
Also, it is known from \cite[Proposition 1.14]{IKO} that the cardinality of $\mathcal{T}$ computed in {\it Step 1} equals $\lfloor p/4 \rfloor$.
There are $\varTheta(p^2)$ possible choices of pairs $(\beta_1,\beta_2)$ in {\it Step 2}, and the probability that the elliptic curve $E_3: y^2 = (x^2-\beta_1)(x^2-\beta_2)$ is supersingular can be estimated to be approximately $1/(2p)$ for each $(\beta_1,\beta_2)$.
Therefore, the expected number of superspecial $Y_{s,t}$ in characteristic $p$ is $\varTheta(p)$.

However, finding the roots of the polynomial $h(\beta)$ is computationally expensive in practice, and therefore Algorithm \ref{alg:pre-genus5} above is not very efficient.
Instead, we explore a way to determine $s$ and $t$ from the $j$-invariants of $E_1$ and $E_2$.
We recall from Proposition~\ref{prop:Case3} that the $j$-invariants of $E_1$ and $E_2$ are given by\vspace{-1.4mm}
\begin{align*}
    j(E_1) = 64\frac{(3c_1-10)^3}{(c_1-2)(c_1+2)^2} &= 64\frac{(3s-1)^3(s-3)^3}{(s-1)^2(s+1)^4},\\[-0.2mm]
    j(E_2) = 64\frac{(3c_2-10)^3}{(c_2-2)(c_2+2)^2} &= 64\frac{(3t-1)^3(t-3)^3}{(t-1)^2(t+1)^4}.\\[-6.7mm]
\end{align*}
Consequently, for given $j_1,j_2 \in \bar{K}$, our curve $Y_{s,t}$ can be constructed by solving the equations\vspace{-1.6mm}
\begin{align}\label{eq:st_for_genus5}
    \begin{split}
        64(3s-1)^3(s-3)^3 - j_1(s-1)^2(s+1)^4 &= 0,\\[-1.4mm]
        64(3t-1)^3(t-3)^3 - j_2(t-1)^2(t+1)^4 &= 0\\[-0.7mm]
    \end{split}
\end{align}
for $s$ and $t$, provided that $s,t \notin\hspace{-0.1mm} \{0,\pm1\}$ and $s^2 \hspace{-0.2mm}\neq t^2$ holds.
In particular, by choosing $j_1$ and $j_2$ as supersingular $j$-invariants, it follows from Proposition~\ref{prop:genus5_decomposition} that $Y_{s,t}$ is superspecial if and only if $E_3$ is supersingular.
Hence, we can produce superspecial $Y_{s,t}$ as follows:\vspace{-1.3mm}
\begin{Alg}\label{alg:genus5}
\hspace{1mm}Input: A prime $p > 5$.\\
\hspace{25.7mm}Output: A pair $(s,t)$ such that $Y_{s,t}$ is superspecial in characteristic $p$, or $\bot$.\vspace{-1.1mm}
\begin{enumerate}
\setlength{\leftskip}{22pt}
\item[{\it Step 1.}\,] \hspace{-0.2mm}Compute the set $\mathcal{S}$ of supersingular $j$-invariants in characteristic $p$.\vspace{-1.3mm}
\item[{\it Step 2.}\,] For each pair $(j_1,j_2) \in \mathcal{S}^2$, solve the equations \eqref{eq:st_for_genus5}, and check whether the elliptic curve \eqref{eq:E3_for_genus5} is supersingular for each solution $(s,t)$ satisfying $s,t \notin\hspace{-0.1mm} \{0,\pm1\}$ and $s^2 \hspace{-0.2mm}\neq t^2$.
If so, output $(s,t)$.\vspace{-1.1mm}
\item[{\it Step 3.}\,] If no such $(s,t)$ is found, output $\bot$.\vspace{-0.5mm}
\end{enumerate}
\end{Alg}
\noindent As discussed above, the expected number of superspecial $Y_{s,t}$ is $\varTheta(p)$, which suggests that Algorithm \ref{alg:genus5} will yield at least one such $Y_{s,t}$ for sufficiently large $p$.
In fact, within the range of our experiments, such a curve is found for all $p > 137$ (see Section~\ref{subsec:experiment_for_genus5} for details).

\section{Genus-6 curves with completely decomposable Jacobians}\label{sec:genus6}
\setcounter{equation}{0}
In this section, we construct a family of generalized Howe curves of genus 6, specifically those whose Jacobians are completely decomposable.
Subsequently, we give an algorithm (Algorithm \ref{alg:genus6}) for generating superspecial curves among them.

Consider two genus-2 curves $C_1$ and $C_2$ sharing exactly $r=3$\hspace{-0.1mm} Weierstrass points:\vspace{-1.4mm}
\begin{align}\label{eq:generalized_genus6}
    \begin{split}
    C_1: y^2 &= (x-a_1)(x-a_2)(x-a_3)(x-b_1)(x-b_2)(x-b_3),\\[-1.5mm]
    C_2: y^2 &= (x-a_1)(x-a_2)(x-a_3)(x-c_1)(x-c_2)(x-c_3),\\[-0.9mm]
    \end{split}
\end{align}
where $a_i,b_i,c_i \in \mathbb{P}^1(\bar{K})$ are all distinct.
The desingularization of the fiber product $C_1 \times_{\mathbb{P}^1}\hspace{-0.1mm} C_2$ is a generalized Howe curve, whose genus is $6$ according to Proposition~\ref{prop:genus}.
In the following, we focus on the case where both genus-2 curves $C_1$ and $C_2$ fall into Case~2 (or its specialization).
More explicitly, for parameters $s,t \in \bar{K}\hspace{-0.3mm}\smallsetminus\hspace{-0.3mm}\{0\}$ such that $1,s^6,t^6$ are pairwise distinct, let $C_1$ and $C_2$ be genus-2 curves defined by\vspace{-1.4mm}
\begin{align*}
        C_1: y^2 &= (x^3-1)(x^3-s^6),\\[-1.5mm]
        C_2: y^2 &= (x^3-1)(x^3-t^6).\\[-6mm]
\end{align*}
We denote by $Z_{s,t}$ the generalized Howe curve constructed from these curves.
By Theorem~\ref{thm:decomposition}, the Jacobian of $Z_{s,t}$ is isogenous to $J(C_1) \hspace{-0.1mm}\times\hspace{-0.1mm} J(C_2) \hspace{-0.1mm}\times\hspace{-0.1mm} J(C_3)$, where $C_3$ is the third genus-2 curve defined by\vspace{-1.3mm}
\[
    C_3: y^2 = (x^3-s^6)(x^3-t^6),\vspace{0.2mm}
\]
which is isomorphic to the genus-2 curve $y^2 = (x^3-1)(x^3-s^6/t^6)$ via the  M\"{o}bius transformation $x \mapsto x/t^2$.
Here, we prepare the following lemma:\vspace{-1.2mm}
\begin{Lem}\label{lem:genus6-Fp2}
If $Z_{s,t}$ is superspecial, then both $s$ and $t$ belong to $\mathbb{F}_{p^2}$.\vspace{-0.8mm}
\end{Lem}
\begin{proof}
Applying the M\"{o}bius transformation $x \mapsto\hspace{-0.2mm} (1-\omega x)/(x-\omega)$ maps the genus-2 curve $C_1$ to its Rosenhain form $y^2 = x(x-1)(x-\lambda_1)(x-\lambda_2)(x-\lambda_3)$ with\vspace{-1.3mm}
\[
    \lambda_1 \coloneqq -\omega\frac{s^2-\omega^2}{s^2-\omega}, \quad \lambda_2 \coloneqq -\omega\frac{s^2-1}{s^2-\omega^2}, \quad \lambda_3 \coloneqq -\omega\frac{s^2-\omega}{s^2-1},\vspace{-1.5mm}
\]
where $\omega$ denotes a primitive cube root of unity (we note that $\omega$ lies in $\mathbb{F}_{p^2}$, since $\omega^{p^2} \hspace{-0.7mm}= \omega \cdot (\omega^3)^{(p^2-1)/3} = \omega$).
By Theorem~\ref{thm:decomposition}, if $Z_{s,t}$ is superspecial, then $C_1$ is also superspecial, which implies that $\lambda_1,\lambda_2$, and $\lambda_3$ all lie in $\mathbb{F}_{p^2}\hspace{-0.5mm}$ by \cite[Main \hspace{-0.2mm}Theorem \hspace{-0.2mm}A]{Ohashi}.
Hence, a direct computation shows that\vspace{-1.7mm}
\[
    s^2 = \frac{\omega^2\hspace{-0.5mm}+\hspace{-0.1mm}\lambda_3}{\omega+\lambda_3} \in \mathbb{F}_{p^2}.\vspace{-0.4mm}
\]
In addition, one can verify that\vspace{-1.5mm}
\begin{align*}
    (1-\lambda_1)\lambda_2 &= \biggl(-\omega^2\frac{s^2-1}{s^2-\omega}\biggr)\hspace{-0.5mm}\biggl(-\omega\frac{s^2-1}{s^2-\omega^2}\biggr) = \frac{(s^2-1)^2}{(s^2-\omega)(s^2-\omega^2)},\\
    \lambda_1-\lambda_2 &= \frac{3s^2}{(s^2-\omega)(s^2-\omega^2)},\\[-5.9mm]
\end{align*}
both of which are squares in $\mathbb{F}_{p^2}\hspace{-0.2mm}$, again by \cite[Main \hspace{-0.2mm}Theorem \hspace{-0.2mm}A]{Ohashi}.
It follows from these equations that\vspace{-1.4mm}
\[
    s^2 = \frac{1}{3}(s^2-1)^2 \cdot \frac{\lambda_1-\lambda_2}{(1-\lambda_1)\lambda_2}.\vspace{-0.7mm}
\]
Since the right-hand side is a square in $\mathbb{F}_{p^2}$,
we conclude that $s \in \mathbb{F}_{p^2}$.
Following the same procedure for $C_2$, we obtain $t \in \mathbb{F}_{p^2}\hspace{-0.2mm}$, which completes the proof.
\end{proof}

According to Proposition~\ref{prop:Case2}, the Jacobian of each $C_i$ is isogenous to $E_i^2$, where\vspace{-1.6mm}
\[
    E_i: y^2 = (c_i+2)x^3 - (3c_i-30)x^2 + (3c_i+30)x - (c_i-2)
\]
with $c_1 \coloneqq s^3 + 1/s^3, \,c_2 \coloneqq t^3 + 1/t^3$, and $c_3 \coloneqq s^3/t^3 + t^3/s^3$ \hspace{-0.2mm}(this is why we define $C_1$ and $C_2$ using $s^6$ and $t^6$\\ instead of $s^3$ and $t^3$; \hspace{0.3mm}to ensure that each $c_i$ is a rational function in $s$ and $t$).
Therefore, the Jacobian of $Z_{s,t}$ is completely decomposable:\vspace{-1.4mm}
\begin{Prop}\label{prop:genus6_decomposition}
With the above notation, there exists an isogeny $\hspace{-0.2mm}J(Z_{s,t}) \rightarrow E_1^2 \times E_2^2 \times E_3^2$ whose degree is\\ a product of powers of 2 and 3.
In particular, the curve $Z_{s,t}$ is superspecial if and only if $E_1,E_2$, and $E_3$ are all supersingular.
\end{Prop}

\newpage
Here, we define the polynomial $g(\alpha)$ by\vspace{-1.5mm}
\begin{equation}\label{eq:g(x)}
    g(\alpha) \coloneqq \sum_{i=0}^{\lfloor p/3 \rfloor} \binom{(p-1)/2}{\lfloor(p+1)/6\rfloor+i}\hspace{-0.5mm}\binom{(p-1)/2}{i}\alpha^i.\vspace{-0.9mm}
\end{equation}
It follows from \cite[Proposition 1.8]{IKO} that the three genus-2 curves $C_1,C_2$, and $C_3$ are superspecial if and only if $g(s^6) = 0,\,g(t^6) = 0$, and $g(s^6/t^6) = 0$, respectively.
Hence, by Theorem~\ref{thm:decomposition}, the curve $Z_{s,t}$ is superspecial if and only if $s^6,t^6$, and $s^6/t^6$ are roots of the polynomial \eqref{eq:g(x)}.
Based on this property, one can construct a {\it naive} algorithm to generate superspecial $Z_{s,t}$ as follows:\vspace{-1.2mm}
\begin{Alg}\label{alg:pre-genus6}
\hspace{1mm}Input: A prime $p > 5$.\\
\hspace{25.7mm}Output: A pair $(s,t)$ such that $Z_{s,t}$ is superspecial in characteristic $p$, or $\bot$.\vspace{-1.1mm}
\begin{enumerate}
\setlength{\leftskip}{22pt}
\item[{\it Step 1.}\,] \hspace{-0.2mm}Compute the set $\mathcal{T}$ of roots $\alpha \neq 0,1$ of the polynomial \eqref{eq:g(x)}.\vspace{-1.2mm}
\item[{\it Step 2.}\,] For each pair $(\alpha_1,\alpha_2) \in \mathcal{T}^2$ with $\alpha_1 \hspace{-0.1mm}\neq \alpha_2$, check whether $\alpha_1/\alpha_2$ belongs to the set $\mathcal{T}$.
If so, output any $(s,t)$ with $s^6 = \alpha_1$ and $t^6 = \alpha_2$.\vspace{-1.1mm}
\item[{\it Step 3.}\,] If no such $(\alpha_1,\alpha_2)$ is found, output $\bot$.\vspace{-0.6mm}
\end{enumerate}
\end{Alg}
\noindent Due to Lemma~\ref{lem:genus6-Fp2}, the outputs $s$ and $t$ (as well as $\alpha_1$ and $\alpha_2$) lie in $\mathbb{F}_{p^2}\hspace{-0.2mm}$.
It follows from \cite[Proposition~1.14]{IKO} that the cardinality of $\mathcal{T}$ computed in {\it Step 1} equals $\lfloor p/3 \rfloor$, which implies that there are $\varTheta(p^2)$ possible choices of pairs $(\alpha_1,\alpha_2)$ in {\it Step 2}.
The probability that $\alpha_1/\alpha_2$ also lies in $\mathcal{T}$ is estimated to be approximately $1/(3p)$ for each $(\alpha_1,\alpha_2)$, and thus the expected number of superspecial $Z_{s,t}$ in characteristic $p$ is $\varTheta(p)$.

However, finding the roots of the polynomial $g(\alpha)$ is computationally expensive in practice, and therefore Algorithm~\ref{alg:pre-genus6} above is not very efficient.
Instead, we describe a way to determine $s$ and $t$ from the $j$-invariants of $E_1$ and $E_2$.
We recall from Proposition~\ref{prop:Case2} that the $j$-invariants of $E_1$ and $E_2$ are given by\vspace{-1.2mm}
\begin{align*}
    j(E_1) = 6912\frac{(2c_1-5)^3}{(c_1-2)(c_1+2)^3} &= 6912\frac{s^3(2s^3-1)^3(s^3-2)^3}{(s^3-1)^2(s^3+1)^6},\\
    j(E_2) = 6912\frac{(2c_2-5)^3}{(c_2-2)(c_2+2)^3} &= 6912\frac{t^3(2t^3-1)^3(t^3-2)^3}{(t^3-1)^2(t^3+1)^6}.\\[-6.8mm]
\end{align*}
Consequently, for given $j_1,j_2 \in \bar{K}$, our curve $Z_{s,t}$ can be constructed by solving the equations\vspace{-1.1mm}
\begin{align}\label{eq:st_for_genus6}
    \begin{split}
    6912s^3(2s^3-1)^3(s^3-2)^3 - j_1(s^3-1)^2(s^3+1)^6 &= 0,\\[-1.2mm]
    6912t^3(2t^3-1)^3(t^3-2)^3 - j_2(t^3-1)^2(t^3+1)^6 &= 0\\[-0.3mm]
    \end{split}
\end{align}
for $s^3$ and $t^3$, provided that $s^3,t^3 \notin \{0,1\}$ and $s^6 \neq t^6$ holds.
Choosing $j_1$ and $j_2$ as supersingular $j$-invariants, it follows from Proposition~\ref{prop:genus6_decomposition} that $Z_{s,t}$ is superspecial if and only if $E_3$ is supersingular.
Also, the $j$-invariant of $E_3$ can be computed from $s^3$ and $t^3$ via the following equation:\vspace{-1mm}
\begin{equation}\label{eq:E3_for_genus6}
    j(E_3) = 6912\frac{(2c_3-5)^3}{(c_3-2)(c_3+2)^3} = 6912\frac{s^3t^3(2s^3-t^3)^3(s^3-2t^3)^3}{(s^3-t^3)^2(s^3+t^3)^6}.
\end{equation}
Therefore, we can generate superspecial $Z_{s,t}$ as follows.
In practice, it is not necessary to determine $s$ and $t$; it suffices to compute only $s^3$ and $t^3$.
For the sake of clarity, however, we describe the following algorithm as outputting $(s,t)$.\vspace{-1mm}
\begin{Alg}\label{alg:genus6}
\hspace{1mm}Input: A prime $p > 5$.\\
\hspace{25.7mm}Output: A pair $(s,t)$ such that $Z_{s,t}$ is superspecial in characteristic $p$, or $\bot$.\vspace{-1.1mm}
\begin{enumerate}
\setlength{\leftskip}{22pt}
\item[{\it Step 1.}\,] \hspace{-0.2mm}Compute the set $\mathcal{S}$ of supersingular $j$-invariants in characteristic $p$.\vspace{-1.3mm}
\item[{\it Step 2.}\,] For each pair $(j_1,j_2) \in \mathcal{S}^2$, solve the equations \eqref{eq:st_for_genus6}, and check whether the value \eqref{eq:E3_for_genus6} belongs to the set $\mathcal{S}$ for each solution $(s,t)$ satisfying $s^3,t^3 \notin\hspace{-0.1mm} \{0,1\}$ and $s^6\hspace{-0.2mm}\neq t^6$.
If so, output $(s,t)$.\vspace{-1.1mm}
\item[{\it Step 3.}\,] If no such $(s,t)$ is found, output $\bot$.\vspace{-0.5mm}
\end{enumerate}
\end{Alg}
\noindent As discussed above, the expected number of superspecial $Z_{s,t}$ is $\varTheta(p)$, which suggests that Algorithm \ref{alg:genus6} will yield at least one such $Z_{s,t}$ for sufficiently large $p$.
In fact, within the range of our experiments, such a curve is found for all $p > 571$ (see Section~\ref{subsec:experiment_for_genus6} for details).

\newpage
\section{Experimental results}\label{sec:experiment}
We implemented Algorithms~\ref{alg:genus4}, \ref{alg:genus5}, and \ref{alg:genus6} as described in the previous sections in \textsf{Magma} Computational Algebra System.
The specific source code is available at the following URL:\vspace{-0.7mm}
\begin{center}
    \url{https://github.com/Ryo-Ohashi/FindSSpGenHowe}.\vspace{1mm}
\end{center}
In this section, we present the experimental results obtained by executing these algorithms.
The experiments were conducted on a machine equipped with an AMD EPYC 7742 CPU and 2TB of RAM.\vspace{-1mm}

\subsection{The existence of superspecial genus-4 curves}\label{subsec:experiment_for_genus4}
We applied Algorithm~\ref{alg:genus4} for all primes $p$ in the range $7 < p < 10^6$ and $p \equiv 1 \pmod{6}$ to  verify the existence of superspecial $X_{s,t}$.
Note that the case $p \equiv 5 \pmod{6}$ can be excluded, as the existence for such $p$ is already guaranteed by Corollary~\ref{cor:existence_genus4}.
This completes the proof of Theorem~\ref{thm:genus4}, which was stated in the Introduction; the total computation time required to establish this result was 206,619 seconds ($\approx$ 57.4 hours).
By combining this theorem with Kudo-Harashita-Howe's result, we conclude the following corollary:\vspace{-0.9mm}
\setcounter{Def}{0}
\begin{Cor}
There exists a superspecial curve of genus $4$ in every characteristic $p$ with $7 < p < 10^6$.\vspace{-0.5mm}
\end{Cor}
\begin{Rmk}
To demonstrate the efficiency of Algorithm~\ref{alg:genus4}, we compared its performance with the method of Kudo-Harashita-Howe.
In the range $7 < p < 20000$, the total execution time of their algorithm was 41,491 seconds ($\approx$ 11.5 hours), whereas our algorithm required only 29.7 seconds to test all the same primes.
While our algorithm failed to output a superspecial genus-4 curve for $p \in\hspace{-0.2mm} \{13,19,73\}$, the time spent on these cases was negligible and had almost no impact on the overall comparison.
Consequently, our algorithm achieves a speedup of approximately 1,400 times compared to the method of Kudo-Harashita-Howe.
\end{Rmk}

\subsection{The existence of superspecial genus-5 curves}\label{subsec:experiment_for_genus5}
Applying Algorithm~\ref{alg:genus5} for all primes $p$ in the range $7 \leq p < 10^5$ to verify the existence of superspecial $Y_{s,t}$, we obtain the following result:\vspace{-1mm}
\begin{Thm}\label{thm:pre-genus5}
For every prime $p$ with $19 < p < 10^5$ and $p \neq\hspace{-0.2mm} 37, 53, 89, 97, 101, 137$, there exists a superspecial curve $Y_{s,t}$ of genus 5 in characteristic $p$. 
\end{Thm}
\noindent This computation took a total of 55,772 seconds ($\approx 15.5$ hours) to establish Theorem~\ref{thm:pre-genus5}.
For the remaining primes $p$ (except for $p=13$), we individually \hspace{-0.2mm}constructed superspecial generalized Howe curves of genus 5, as\\ described at the beginning of Section \ref{sec:genus5}.
The explicit defining equations for the genus-2 curves $C_1,C_2$ in \eqref{eq:generalized_genus5} found for each $p$ are provided in Appendix~\ref{app:genus5}.
By combining these results, we finally arrive at Theorem \ref{thm:genus5}.

\subsection{The existence of superspecial genus-6 curves}\label{subsec:experiment_for_genus6}Applying Algorithm~\ref{alg:genus6} for all primes $p$ in the range $7 \leq p < 10^5$ to verify the existence of superspecial $Z_{s,t}$, we obtain the following result:\vspace{-1mm}
\begin{Thm}\label{thm:pre-genus6}
For every prime $p$ with $11 < p < 10^5$ and\vspace{-0.8mm}
\[
    p \neq 19, 37, 43, 61, 67, 79, 97, 109, 127, 151, 157, 223, 229, 283, 313, 331, 337, 373, 571,\vspace{0.2mm}
\]
there exists a superspecial curve $Z_{s,t}$ of genus 6 in characteristic $p$. \vspace{-0.7mm}
\end{Thm}
\noindent This computation took a total of 79,666 seconds ($\approx 22.1$ hours) to establish Theorem~\ref{thm:pre-genus6}.
For the remaining primes $p$ with $p>7$, we individually \hspace{-0.2mm}constructed superspecial generalized Howe curves of genus 6, as described at the beginning of Section \ref{sec:genus6}.
The explicit defining equations for the genus-2 curves $C_1,C_2$ in \eqref{eq:generalized_genus6} found for each $p$ are provided in Appendix~\ref{app:genus6}.
By combining these results, we finally arrive at Theorem \ref{thm:genus6}.\vspace{-1mm}

\paragraph*{Acknowledgements.}
This research was supported by JSPS Grant-in-Aid for Young Scientists 25K17225.

\vspace{2.3mm}

\textsc{Graduate School of Information Science and Technology, The
University of Tokyo — 7-3-1 Hongo, Bunkyo-ku, Tokyo, 113-0033, Japan.}\par
\textit{E-mail address}: \url{ryo-ohashi@g.ecc.u-tokyo.ac.jp}

\newpage
\appendix
\section{Constructions for characteristics excluded in Theorem 6.3}\label{app:genus5}
In this appendix, we provide examples of superspecial generalized Howe curves of genus 5 for characteristic $p$ with $p \in\hspace{-0.2mm} \{7,11,17,19,37,53,89,97,101,137\}$.
For the construction of examples, we first collect all Rosenhain forms of superspecial genus-2 curves, following Step 1 of \cite[Algorithm 3.1.2]{OK}.
Next, we select pairs $(C_1,C_2)$ among them sharing exactly four \hspace{-0.4mm}Weierstrass points, given by\vspace{-1.4mm}
\begin{align*}
    C_1: y^2 &= x(x-1)(x-\lambda_1)(x-\lambda_2)(x-\lambda_3),\\[-1.5mm]
    C_2: y^2 &= x(x-1)(x-\lambda_1)(x-\lambda'_2)(x-\lambda'_3),\\[-5.2mm]
\end{align*}
where $\lambda_1,\lambda_2,\lambda_3,\lambda'_2,\lambda'_3$ are pairwise distinct.
We then determine whether the elliptic curve defined by\vspace{-1.3mm}
\[
    E_3 : y^2 = (x-\lambda_2)(x-\lambda_3)(x-\lambda'_2)(x-\lambda'_3)
\]
is supersingular.
For each characteristic $p$ listed above, such a pair $(C_1, C_2)$ of superspecial genus-2 curves is found, and is presented below:\vspace{-0.9mm}
\begin{itemize}
    \item For $p=7$, we found a pair $(C_1,C_2)$ given by\vspace{-1.2mm}
    \begin{align*}
        C_1: y^2 &= x(x-1)(x-\zeta^{34})(x-\zeta^2)(x-\zeta^{36}),\\[-1.5mm]
        C_2: y^2 &= x(x-1)(x-\zeta^{34})(x-4)(x-\zeta^{46}),
    \end{align*}
    where $\zeta \in \mathbb{F}_{7^2}\hspace{-0.2mm}$ is a root of the quadratic equation $z^2 - z + 3 = 0$.\vspace{-0.7mm}
    \item For $p=11$, we found a pair $(C_1,C_2)$ given by\vspace{-1.2mm}
    \begin{align*}
        C_1: y^2 &= x(x-1)(x-\zeta^2)(x-\zeta^{80})(x-\zeta^{98}),\\[-1.5mm]
        C_2: y^2 &= x(x-1)(x-\zeta^2)(x-\zeta^{86})(x-\zeta^{88}),
    \end{align*}
    where $\zeta \in \mathbb{F}_{11^2}\hspace{-0.2mm}$ is a root of the quadratic equation $z^2 - 4z + 2 = 0$.\vspace{-0.7mm}
    \item For $p=17$, we found a pair $(C_1,C_2)$ given by\vspace{-1.2mm}
    \begin{align*}
        C_1: y^2 &= x(x-1)(x-16)(x-\zeta^{194})(x-\zeta^{238}),\\[-1.5mm]
        C_2: y^2 &= x(x-1)(x-16)(x-\zeta^{50})(x-\zeta^{94}),
    \end{align*}
    where $\zeta \in \mathbb{F}_{17^2}\hspace{-0.2mm}$ is a root of the quadratic equation $z^2 - z + 3 = 0$.\vspace{-0.7mm}
    \item For $p=19$, we found a pair $(C_1,C_2)$ given by\vspace{-1.2mm}
    \begin{align*}
        C_1: y^2 &= x(x-1)(x-3)(x-17)(x-10),\\[-1.5mm]
        C_2: y^2 &= x(x-1)(x-3)(x-\zeta^{34})(x-\zeta^{302}),
    \end{align*}
    where $\zeta \in \mathbb{F}_{19^2}\hspace{-0.2mm}$ is a root of the quadratic equation $z^2 - z + 2 = 0$.\vspace{-0.7mm}
    \item For $p=37$, we found a pair $(C_1,C_2)$ given by\vspace{-1.2mm}
    \begin{align*}
        C_1: y^2 &= x(x-1)(x-\zeta^{1306})(x-\zeta^{156})(x-\zeta^{574}),\\[-1.5mm]
        C_2: y^2 &= x(x-1)(x-\zeta^{1306})(x-\zeta^{584})(x-\zeta^{1138}),
    \end{align*}
    where $\zeta \in \mathbb{F}_{37^2}\hspace{-0.2mm}$ is a root of the quadratic equation $z^2 - 4z + 2 = 0$.\vspace{-0.7mm}
    \item For $p=53$, we found a pair $(C_1,C_2)$ given by\vspace{-1.2mm}
    \begin{align*}
        C_1: y^2 &= x(x-1)(x-\zeta^{2})(x-\zeta^{538})(x-\zeta^{1288}),\\[-1.5mm]
        C_2: y^2 &= x(x-1)(x-\zeta^{2})(x-\zeta^{220})(x-\zeta^{2590}),
    \end{align*}
    where $\zeta \in \mathbb{F}_{53^2}\hspace{-0.2mm}$ is a root of the quadratic equation $z^2 - 4z + 2 = 0$.\vspace{-0.7mm}
    \item For $p=89$, we found a pair $(C_1,C_2)$ given by\vspace{-1.2mm}
    \begin{align*}
        C_1: y^2 &= x(x-1)(x-\zeta^{7220})(x-66)(x-\zeta^{1646}),\\[-1.5mm]
        C_2: y^2 &= x(x-1)(x-\zeta^{7220})(x-\zeta^{4722})(x-\zeta^{5222}),
    \end{align*}
    where $\zeta \in \mathbb{F}_{89^2}\hspace{-0.2mm}$ is a root of the quadratic equation $z^2 - 7z + 3 = 0$.\vspace{-0.7mm}
    \item For $p=97$, we found a pair $(C_1,C_2)$ given by\vspace{-1.2mm}
    \begin{align*}
        C_1: y^2 &= x(x-1)(x-\zeta^{2})(x-\zeta^{4030})(x-\zeta^{6088}),\\[-1.5mm]
        C_2: y^2 &= x(x-1)(x-\zeta^{2})(x-\zeta^{3442})(x-\zeta^{5808}),
    \end{align*}
    where $\zeta \in \mathbb{F}_{97^2}\hspace{-0.2mm}$ is a root of the quadratic equation $z^2 - z + 5 = 0$.\vspace{-0.7mm}
    \item For $p=101$, we found a pair $(C_1,C_2)$ given by\vspace{-1.2mm}
    \begin{align*}
        C_1: y^2 &= x(x-1)(x-\zeta^{7214})(x-\zeta^{2882})(x-\zeta^{7700}),\\[-1.5mm]
        C_2: y^2 &= x(x-1)(x-\zeta^{7214})(x-\zeta^{4920})(x-\zeta^{6838}),
    \end{align*}
    where $\zeta \in \mathbb{F}_{101^2}\hspace{-0.2mm}$ is a root of the quadratic equation $z^2 - 4z + 2 = 0$.\vspace{-0.7mm}
    \item For $p=137$, we found a pair $(C_1,C_2)$ given by\vspace{-1.2mm}
    \begin{align*}
        C_1: y^2 &= x(x-1)(x-\zeta^{12758})(x-\zeta^{2582})(x-\zeta^{12468}),\\[-1.5mm]
        C_2: y^2 &= x(x-1)(x-\zeta^{12758})(x-130)(x-\zeta^{15614}),
    \end{align*}
    where $\zeta \in \mathbb{F}_{137^2}\hspace{-0.2mm}$ is a root of the quadratic equation $z^2 - 6z + 3 = 0$.\vspace{-0.7mm}
\end{itemize}
This completes the construction of superspecial genus-5 curves for the characteristics $p$ remaining in the proof of Theorem \ref{thm:genus5}.\vspace{-2mm}

\section{Constructions for characteristics excluded in Theorem 6.4}\label{app:genus6}
In this appendix, we provide examples of superspecial generalized Howe curves of genus 6 for characteristic $p$\\ with $p \in\hspace{-0.2mm} \{11,19,37,43,61,67,79,97,109,127,151,157,223,229,283,313,331,337,373,571\}$.
In a similar way to Appendix~\ref{app:genus5}, after listing all Rosenhain forms of genus 2 curves, we select pairs $(C_1, C_2)$ that share exactly three \hspace{-0.4mm}Weierstrass points, given by\vspace{-0.8mm}
\begin{align*}
    C_1: y^2 &= x(x-1)(x-\lambda_1)(x-\lambda_2)(x-\lambda_3),\\[-1.5mm]
    C_2: y^2 &= x(x-1)(x-\lambda'_1)(x-\lambda'_2)(x-\lambda'_3),\\[-5.4mm]
\end{align*}
where $\lambda_1,\lambda_2,\lambda_3,\lambda'_1,\lambda'_2,\lambda'_3$ are pairwise distinct elements of $\mathbb{F}_{p^2}\hspace{-0.3mm}$ other than $0$ and $1$.
Subsequently, we determine whether the third genus-2 curve defined by the equation\vspace{-1.2mm}
\[
    C_3: y^2 = (x-\lambda_1)(x-\lambda_2)(x-\lambda_3)(x-\lambda'_1)(x-\lambda'_2)(x-\lambda'_3)
\]
is also superspecial.
For each characteristic $p$ listed above, such a pair $(C_1,C_2)$ of superspecial genus-2 curves is found, and is presented below:\vspace{-0.9mm}
\begin{itemize}
    \item For $p=11$, we found a pair $(C_1,C_2)$ given by\vspace{-1.2mm}
    \begin{align*}
        C_1: y^2 &= x(x-1)(x-2)(x-7)(x-3),\\[-1.5mm]
        C_2: y^2 &= x(x-1)(x-5)(x-10)(x-9),\\[-5.2mm]
    \end{align*}
    which are both defined over $\mathbb{F}_{11}$.\vspace{-0.7mm}
    \item For $p=19$, we found a pair $(C_1,C_2)$ given by\vspace{-1.2mm}
    \begin{align*}
        C_1: y^2 &= x(x-1)(x-\zeta^{50})(x-\zeta^{250})(x-\zeta^{332}),\\[-1.5mm]
        C_2: y^2 &= x(x-1)(x-\zeta^{130})(x-\zeta^{172})(x-\zeta^{290}),\\[-5.2mm]
    \end{align*}
    where $\zeta \in \mathbb{F}_{19^2}\hspace{-0.2mm}$ is a root of the quadratic equation $z^2 - z + 2 = 0$.\vspace{-0.7mm}
    \item For $p=37$, we found a pair $(C_1,C_2)$ given by\vspace{-1.2mm}
    \begin{align*}
        C_1: y^2 &= x(x-1)(x-15)(x-5)(x-\zeta^{1044}),\\[-1.5mm]
        C_2: y^2 &= x(x-1)(x-\zeta^{62})(x-\zeta^{252})(x-\zeta^{442}),\\[-5.2mm]
    \end{align*}
    where $\zeta \in \mathbb{F}_{37^2}\hspace{-0.2mm}$ is a root of the quadratic equation $z^2 - 4z + 2 = 0$.\vspace{-0.7mm} 
    \item For $p=43$, we found a pair $(C_1,C_2)$ given by\vspace{-1.2mm}
    \begin{align*}
        C_1: y^2 &= x(x-1)(x-\zeta^{224})(x-\zeta^{1104})(x-\zeta^{1524}),\\[-1.5mm]
        C_2: y^2 &= x(x-1)(x-\zeta^{232})(x-\zeta^{252})(x-\zeta^{780}),\\[-5.2mm]
    \end{align*}
    where $\zeta \in \mathbb{F}_{43^2}\hspace{-0.2mm}$ is a root of the quadratic equation $z^2 - z + 3 = 0$.\vspace{-0.7mm}\newpage
    \item For $p=61$, we found a pair $(C_1,C_2)$ given by\vspace{-1.2mm}
    \begin{align*}
        C_1: y^2 &= x(x-1)(x-\zeta^{483})(x-\zeta^{1974})(x-\zeta^{3164}),\\[-1.5mm]
        C_2: y^2 &= x(x-1)(x-\zeta^{114})(x-\zeta^{436})(x-\zeta^{2346}),\\[-5.4mm]
    \end{align*}
    where $\zeta \in \mathbb{F}_{61^2}\hspace{-0.2mm}$ is a root of the quadratic equation $z^2 - z + 2 = 0$.\vspace{-0.7mm}
    \item For $p=67$, we found a pair $(C_1,C_2)$ given by\vspace{-1.2mm}
    \begin{align*}
        C_1: y^2 &= x(x-1)(x-\zeta^{1148})(x-\zeta^{3734})(x-\zeta^{4346}),\\[-1.5mm]
        C_2: y^2 &= x(x-1)(x-\zeta^{538})(x-\zeta^{1150})(x-\zeta^{3868}),\\[-5.4mm]
    \end{align*}
    where $\zeta \in \mathbb{F}_{67^2}\hspace{-0.2mm}$ is a root of the quadratic equation $z^2 - 4z + 2 = 0$.\vspace{-0.7mm}
    \item For $p=79$, we found a pair $(C_1,C_2)$ given by\vspace{-1.2mm}
    \begin{align*}
        C_1: y^2 &= x(x-1)(x-\zeta^{2118})(x-71)(x-\zeta^{5082}),\\[-1.5mm]
        C_2: y^2 &= x(x-1)(x-32)(x-35)(x-7),\\[-5.4mm]
    \end{align*}
    where $\zeta \in \mathbb{F}_{79^2}\hspace{-0.2mm}$ is a root of the quadratic equation $z^2 - z + 3 = 0$.\vspace{-0.7mm}
    \item For $p=97$, we found a pair $(C_1,C_2)$ given by\vspace{-1.2mm}
    \begin{align*}
        C_1: y^2 &= x(x-1)(x-\zeta^{2292})(x-\zeta^{3468})(x-\zeta^{7968}),\\[-1.5mm]
        C_2: y^2 &= x(x-1)(x-\zeta^{12})(x-\zeta^{7776})(x-\zeta^{8244}),\\[-5.4mm]
    \end{align*}
    where $\zeta \in \mathbb{F}_{97^2}\hspace{-0.2mm}$ is a root of the quadratic equation $z^2 - z + 5 = 0$.\vspace{-0.7mm}
    \item For $p=109$, we found a pair $(C_1,C_2)$ given by\vspace{-1.2mm}
    \begin{align*}
        C_1: y^2 &= x(x-1)(x-\zeta^{1476})(x-\zeta^{1860})(x-\zeta^{2136}),\\[-1.5mm]
        C_2: y^2 &= x(x-1)(x-\zeta^{4584})(x-\zeta^{10908})(x-\zeta^{11184}),\\[-5.4mm]
    \end{align*}
    where $\zeta \in \mathbb{F}_{109^2}\hspace{-0.2mm}$ is a root of the quadratic equation $z^2 - z + 6 = 0$.\vspace{-0.7mm}
    \item For $p=127$, we found a pair $(C_1,C_2)$ given by\vspace{-1.2mm}
    \begin{align*}
        C_1: y^2 &= x(x-1)(x-\zeta^{502})(x-\zeta^{9172})(x-\zeta^{13058}),\\[-1.5mm]
        C_2: y^2 &= x(x-1)(x-\zeta^{1440})(x-\zeta^{1874})(x-\zeta^{15370}),\\[-5.4mm]
    \end{align*}
    where $\zeta \in \mathbb{F}_{127^2}\hspace{-0.2mm}$ is a root of the quadratic equation $z^2 - z + 3 = 0$.\vspace{-0.7mm}
    \item For $p=151$, we found a pair $(C_1,C_2)$ given by\vspace{-1.2mm}
    \begin{align*}
        C_1: y^2 &= x(x-1)(x-\zeta^{10800})(x-\zeta^{13500})(x-\zeta^{18900}),\\[-1.5mm]
        C_2: y^2 &= x(x-1)(x-\zeta^{600})(x-\zeta^{1500})(x-\zeta^{4500}),\\[-5.4mm]
    \end{align*}
    where $\zeta \in \mathbb{F}_{151^2}\hspace{-0.2mm}$ is a root of the quadratic equation $z^2 - 2z + 6 = 0$.\vspace{-0.7mm}
    \item For $p=157$, we found a pair $(C_1,C_2)$ given by\vspace{-1.2mm}
    \begin{align*}
        C_1: y^2 &= x(x-1)(x-\zeta^{66})(x-\zeta^{7152})(x-\zeta^{14918}),\\[-1.5mm]
        C_2: y^2 &= x(x-1)(x-\zeta^{13262})(x-\zeta^{14210})(x-\zeta^{21296}),\\[-5.4mm]
    \end{align*}
    where $\zeta \in \mathbb{F}_{157^2}\hspace{-0.2mm}$ is a root of the quadratic equation $z^2 - 5z + 5 = 0$.\vspace{-0.7mm}
    \item For $p=223$, we found a pair $(C_1,C_2)$ given by\vspace{-1.2mm}
    \begin{align*}
        C_1: y^2 &= x(x-1)(x-\zeta^{37056})(x-\zeta^{40200})(x-\zeta^{44956}),\\[-1.5mm]
        C_2: y^2 &= x(x-1)(x-\zeta^{13330})(x-\zeta^{38554})(x-\zeta^{40216}),\\[-5.4mm]
    \end{align*}
    where $\zeta \in \mathbb{F}_{223^2}\hspace{-0.2mm}$ is a root of the quadratic equation $z^2 - 2z + 3 = 0$.\vspace{-0.7mm}\newpage
    \item For $p=229$, we found a pair $(C_1,C_2)$ given by\vspace{-1.2mm}
    \begin{align*}
        C_1: y^2 &= x(x-1)(x-\zeta^{23644})(x-\zeta^{32210})(x-\zeta^{38638}),\\[-1.5mm]
        C_2: y^2 &= x(x-1)(x-\zeta^{3610})(x-\zeta^{31024})(x-\zeta^{39590}),\\[-5.4mm]
    \end{align*}
    where $\zeta \in \mathbb{F}_{229^2}\hspace{-0.2mm}$ is a root of the quadratic equation $z^2 - z + 6 = 0$.\vspace{-0.7mm}
    \item For $p=283$, we found a pair $(C_1,C_2)$ given by\vspace{-1.2mm}
    \begin{align*}
        C_1: y^2 &= x(x-1)(x-186)(x-149)(x-271),\\[-1.5mm]
        C_2: y^2 &= x(x-1)(x-165)(x-19)(x-35),\\[-5.4mm]
    \end{align*}
    where are both defined over $\mathbb{F}_{283}$.\vspace{-0.7mm}
    \item For $p=313$, we found a pair $(C_1,C_2)$ given by\vspace{-1.2mm}
    \begin{align*}
        C_1: y^2 &= x(x-1)(x-\zeta^{32774})(x-\zeta^{44078})(x-\zeta^{97436}),\\[-1.5mm]
        C_2: y^2 &= x(x-1)(x-\zeta^{2108})(x-\zeta^{32100})(x-\zeta^{52824}),\\[-5.4mm]
    \end{align*}
    where $\zeta \in \mathbb{F}_{313^2}\hspace{-0.2mm}$ is a root of the quadratic equation $z^2 - 3z + 10 = 0$.\vspace{-0.7mm}
    \item For $p=331$, we found a pair $(C_1,C_2)$ given by\vspace{-1.2mm}
    \begin{align*}
        C_1: y^2 &= x(x-1)(x-\zeta^{45614})(x-\zeta^{59208})(x-\zeta^{83986}),\\[-1.5mm]
        C_2: y^2 &= x(x-1)(x-\zeta^{16352})(x-\zeta^{78768})(x-\zeta^{79248}),\\[-5.4mm]
    \end{align*}
    where $\zeta \in \mathbb{F}_{331^2}\hspace{-0.2mm}$ is a root of the quadratic equation $z^2 - 5z + 3 = 0$.\vspace{-0.7mm}
    \item For $p=337$, we found a pair $(C_1,C_2)$ given by\vspace{-1.2mm}
    \begin{align*}
        C_1: y^2 &= x(x-1)(x-\zeta^{13218})(x-\zeta^{60362})(x-\zeta^{83388}),\\[-1.5mm]
        C_2: y^2 &= x(x-1)(x-\zeta^{12942})(x-\zeta^{44456})(x-\zeta^{44962}),\\[-5.4mm]
    \end{align*}
    where $\zeta \in \mathbb{F}_{337^2}\hspace{-0.2mm}$ is a root of the quadratic equation $z^2 - 5z + 10 = 0$.\vspace{-0.7mm}
    \item For $p=373$, we found a pair $(C_1,C_2)$ given by\vspace{-1.2mm}
    \begin{align*}
        C_1: y^2 &= x(x-1)(x-\zeta^{26442})(x-\zeta^{58480})(x-\zeta^{125794}),\\[-1.5mm]
        C_2: y^2 &= x(x-1)(x-\zeta^{22264})(x-\zeta^{26406})(x-\zeta^{53334}),\\[-5.4mm]
    \end{align*}
    where $\zeta \in \mathbb{F}_{373^2}\hspace{-0.2mm}$ is a root of the quadratic equation $z^2 - 4z + 2 = 0$.\vspace{-0.7mm}
    \item For $p=571$, we found a pair $(C_1,C_2)$ given by\vspace{-1.2mm}
    \begin{align*}
        C_1: y^2 &= x(x-1)(x-\zeta^{53462})(x-\zeta^{296376})(x-\zeta^{301634}),\\[-1.5mm]
        C_2: y^2 &= x(x-1)(x-\zeta^{181188})(x-\zeta^{188806})(x-\zeta^{190810}),\\[-5.4mm]
    \end{align*}
    where $\zeta \in \mathbb{F}_{571^2}\hspace{-0.2mm}$ is a root of the quadratic equation $z^2 - z + 3 = 0$.\vspace{-0.7mm}
\end{itemize}
This completes the construction of superspecial genus-6 curves for the characteristics $p$ remaining in the proof of Theorem \ref{thm:genus6}.
\end{document}